\documentclass{svproc}
\usepackage{url}

\usepackage{amsmath,amssymb,bm,cite}
\usepackage[dvipdfmx]{graphicx}

\def\BC{\mathbb C}

\def\BR{\mathbb R}
\def\cA{\mathcal A}

\def\cD{\mathcal D}

\def\rd{\mathrm d}

\def\Ga{\Gamma}

\def\Om{\Omega}
\def\al{\alpha}
\def\be{\beta}
\def\ga{\gamma}

\def\te{\theta}

\def\ka{\kappa}
\def\la{\lambda}

\def\vp{\varphi}

\def\f{\frac}
\def\nb{\nabla}
\def\ov{\overline}
\def\pa{\partial}

\def\tri{\triangle}

\allowdisplaybreaks[4]

\begin{document}
\mainmatter              
\title{Long-time asymptotic estimate and a related\\
inverse source problem for time-fractional\\
wave equations}
\titlerunning{Asymptotic estimate and an inverse problem for fractional wave equations}  
%
\author{Xinchi Huang\inst{1} \and Yikan Liu\inst{2}}
\authorrunning{Xinchi Huang et al.} 
%
\tocauthor{Xinchi Huang, and Yikan Liu}
\institute{Graduate School of Mathematical Sciences, The University of Tokyo, 3-8-1 Komaba, Meguro-ku, Tokyo 153-8914, Japan,\\
\email{huangxc@ms.u-tokyo.ac.jp}
\and
Research Center of Mathematics for Social Creativity, Research Institute for Electronic Science, Hokkaido University, N12W7, Kita-Ward, Sapporo 060-0812, Japan,\\
\email{ykliu@es.hokudai.ac.jp}
}

\maketitle              

\begin{abstract}
Lying between traditional parabolic and hyperbolic equations, time-fractional wave equations of order $\al\in(1,2)$ in time inherit both decaying and oscillating properties. In this article, we establish a long-time asymptotic estimate for homogeneous time-fractional wave equations, which readily implies the strict positivity/negativity of the solution for $t\gg1$ under some sign conditions on initial values. As a direct application, we prove the uniqueness for a related inverse source problem on determining the temporal component.

\keywords{time-fractional wave equation, asymptotic estimate, inverse source problem, uniqueness}
\end{abstract}


\section{Introduction}\label{sec-intro}

Recent several decades have witnessed the explosive development of nonlocal models based on fractional calculus from various backgrounds. Remarkably, between the fundamental equations of elliptic, parabolic and hyperbolic types, partial differential equations (PDEs) like
\begin{equation}\label{eq-TFE}
(\pa_t^\al-\tri)u=F
\end{equation}
with fractional orders $\al\in(0,1)\cup(1,2)$ of time derivatives have attracted interests from both theoretical and applied sides (the meaning of $\pa_t^\al$ will be specified later). Such time-fractional PDEs have been reported to be capable of describing such phenomena as anomalous diffusion in heterogenous medium and viscoelastic materials that usual PDEs fail to describe (e.g. \cite{BP88,BDES18,HH98}).

Due to the similarity with their integer counterparts, equations like \eqref{eq-TFE} are called time-fractional diffusion equations for $\al\in(0,1)$, while are called time-fractional wave ones for $\al\in(1,2)$. In the last decade, modern mathematical theories have been introduced in the study of time-fractional PDEs, and fruitful results on the well-posedness and important properties of solutions have been established especially for $\al\in(0,1)$ (e.g. \cite{EK04,GLY,KRY,SY} and the references therein). On the contrary, time-fractional wave equations (i.e., \eqref{eq-TFE} for $\al\in(1,2)$) seem not well investigated especially from the viewpoint of their relation with the cases of $\al=1$ and $\al=2$. Meanwhile, many related inverse problems remain open.

In the sequel, let $\al\in(1,2)$, $T>0$ be constants and $\Om\subset\BR^d$ ($d=1,2,\ldots$) be a bounded domain whose boundary $\pa\Om$ is sufficiently smooth. The main focuses of this paper are the following two initial-boundary value problems for homogeneous and inhomogeneous time-fractional wave equations:
\begin{equation}\label{eq-ibvp-u0}
\begin{cases}
(\pa_t^\al+\cA)u=0 & \mbox{in }\Om\times(0,T),\\
u=u_0,\ \pa_t u=u_1 & \mbox{in }\Om\times\{0\},\\
u=0 & \mbox{on }\pa\Om\times(0,T)
\end{cases}
\end{equation}
and
\begin{equation}\label{eq-ibvp-u1}
\begin{cases}
(\pa_t^\al+\cA)u(\bm x,t)=\rho(t)f(\bm x), & (\bm x,t)\in\Om\times(0,T),\\
u=\pa_t u=0 & \mbox{in }\Om\times\{0\},\\
u=0 & \mbox{on }\pa\Om\times(0,T).
\end{cases}
\end{equation}
Here, $\pa_t^\al$ denotes the Caputo derivative in the time variable $t>0$ and $\cA$ is a symmetric elliptic operator in the space variable $\bm x\in\Om$, whose definitions will be provided in Section~\ref{sec-premain} in detail. In the homogeneous problem \eqref{eq-ibvp-u0}, there is no external force and $u_0$, $u_1$ stand for the initial displacement and velocity, respectively. In the inhomogeneous problem \eqref{eq-ibvp-u1}, initial displacement and velocity vanish and the source term takes the form of separated variables, where $\rho(t)$ and $f(\bm x)$ stand for the temporal and spatial components, respectively. In both problems \eqref{eq-ibvp-u0}--\eqref{eq-ibvp-u1}, we impose the homogeneous Dirichlet boundary condition, which can be replaced by homogeneous Neumann or Robin ones.

There are some results on the well-posedness and the vanishing property of \eqref{eq-ibvp-u0}--\eqref{eq-ibvp-u1} in \cite{HY22,LHY21,SY}, which basically inherit those for time-fractional diffusion equations. However, the strong positivity property for $0<\al<1$ (see \cite{L17,LRY}) no longer holds for $1<\al<2$, whose solutions oscillate and change signs in general even with strictly positive initial values. Therefore, we are interested in the sign change of the solution to \eqref{eq-ibvp-u0}.

Indeed, we acquire some hints from the graphs of Mittag-Leffler functions $t^{j-1}E_{\al,j}(-t^\al)$ with $j=0,1$ (see \eqref{eq-def-ML} for a definition) in Figure~\ref{fig-ML}, which are closely related to the solution to \eqref{eq-ibvp-u0}. First, the numbers of sign changes for both functions seem to be finite and monotonely increasing with respect to $\al<2$, and as $\al$ reaches $2$ we know $E_{2,1}(-t^2)=\sin t$ and $t\,E_{2,2}(-t^2)=\cos t$. Next, though both functions tends to $0$ as $t\to\infty$, we see that $E_{\al,1}(-t^\al)<0$ and $t\,E_{\al,2}(-t^\al)>0$ after sufficiently large $t$. Since $E_{\al,1}(-t^\al)$ and $t\,E_{\al,2}(-t^\al)$ coincide with the solutions to \eqref{eq-ibvp-u0} at $x=\pi/2$ with $\Om=(0,\pi)$, $\cA=-\pa_x^2$ and the special choice of
\[
u_0(x)=\sin x,\ u_1(x)=0\quad\mbox{and}\quad u_0(x)=0,\ u_1(x)=\sin x,
\]
respectively, we are concerned with such long-time strict positivity/negativity for \eqref{eq-ibvp-u0} with more general initial values.
\begin{figure}[htbp]\label{fig-ML}
\includegraphics[trim=60mm 15mm 60mm 10mm,clip=true,width=\textwidth]{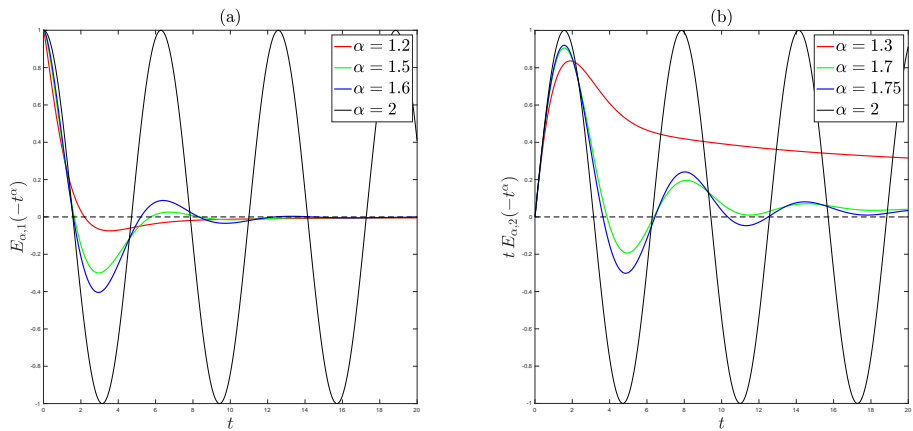}
\caption{Plots of Mittag-Leffler functions $E_{\al,1}(-t^\al)$ (a) and $t\,E_{\al,2}(-t^\al)$ (b) with several choices of $\al\in(1,2]$.}
\end{figure}

As a related topic, we are also interested in the following inverse problem.

\begin{problem}[inverse source problem]\label{isp}
Fix $\bm x_0\in\Om$ and let $u$ be the solution to \eqref{eq-ibvp-u1}. Provided that the spatial component $f$ of the source term is suitably given, determine the temporal component $\rho$ by the single point observation of $u$ at $\{\bm x_0\}\times(0,T)$.
\end{problem}

As before, there is abundant literature on inverse source problem for time-fractional diffusion equations (see \cite{LLY} for a survey), but much less on that for time-fractional wave ones. Moreover, the majority of the latter were treated by uniform methodologies for $0<\al\le2$. We refer to \cite{HLY20,LHY21} for inverse moving source problems, and \cite{KLY} for the inverse source problem on determining $f(\bm x)$ in \eqref{eq-ibvp-u1}. For Problem~\ref{isp}, there are results on uniqueness and stability for $0<\al<1$ (see \cite{L17,LRY,LZ17}), which heavily rely on the strong positivity property of the homogeneous problem \eqref{eq-ibvp-u0}. Thus, we shall consider Problem~\ref{isp} at most with the long-time positivity suggested above for $1<\al<2$.

The remainder of this article is organized as follows. Preparing necessary notations and definitions, in Section~\ref{sec-premain} we state the main results on the long-time asymptotic estimate, strict positivity/negativity and the uniqueness for Problem~\ref{isp}. In Section~\ref{sec-forward}, we show the well-posedness for \eqref{eq-ibvp-u0}--\eqref{eq-ibvp-u1} and establish a fractional Duhamel's principle between them. Then Section~\ref{sec-proof} is devoted to the proof of main results, followed by a brief conclusion in Section~\ref{sec-remark}.


\section{Preliminary and Statement of Main Results}\label{sec-premain}

We start with the definition of the Caputo derivative $\pa_t^\al$ in \eqref{eq-ibvp-u0}--\eqref{eq-ibvp-u1}. Recall the Reimann-Liouville integral operator of order $\be>0$:
\[
J^\be g(t):=\f1{\Ga(\be)}\int_0^t(t-s)^{\be-1}g(s)\,\rd s,\quad g\in C[0,\infty),
\]
where $\Ga(\,\cdot\,)$ is the Gamma function. Then for $1<\al<2$, the pointwise Caputo derivative $\pa_t^\al$ is defined as (e.g. Podlubny \cite{Po})
\[
\pa_t^\al g(t):=\left(J^{2-\al}\circ\f{\rd^2}{\rd t^2}\right)g(t)=\f1{\Ga(2-\al)}\int_0^t\f{g''(s)}{(t-s)^{\al-1}}\,\rd s,\quad g\in C^2[0,\infty),
\]
where $\circ$ is the composite. Notice that the above definition is naive in the sense that it is only valid for smooth functions. For non-smooth functions, recent years there is a modern definition of $\pa_t^\be$ as the inverse of $J^\be$ in the fractional Sobolev space $H_\be(0,T)$ (see e.g. Gorenflo, Luchko and Yamamoto \cite{GLY} for the case of $0<\be<1$ and Huang and Yamamoto \cite{HY22} for that of $1<\be<2$). For instance, the problem \eqref{eq-ibvp-u0} should be formulated e.g. as
\[
\begin{cases}
\pa_t^\al(u-u_0-t\,u_1)+\cA u=0 & \mbox{in }L^2(0,T;H^{-1}(\Om)),\\
u(\,\cdot\,,t)\in H_0^1(\Om), & 0<t<T,\\
u-u_0-t\,u_1\in H_\al(0,T;H^{-1}(\Om))
\end{cases}
\]
in that context. Nevertheless, since the definition of $\pa_t^\al$ is not the main concern of this article, we prefer the traditional formulations \eqref{eq-ibvp-u0}--\eqref{eq-ibvp-u1} for better readability.

Next, we invoke the familiar Mittag-Leffler functions for later use:
\begin{equation}\label{eq-def-ML}
E_{\al,\be}(z):=\sum_{k=0}^\infty\f{z^k}{\Ga(\al k+\be)},\quad z\in\BC,\ \be>0.
\end{equation}
We collect several frequently used estimates for $E_{\al,\be}(z)$.

\begin{lemma}\label{lem-asymp}
Let $\al\in(1,2)$. Then there exist constants $C_0>0$ and $C_0'>0$ depending only on $\al$ such that
\begin{gather}
|E_{\al,\be}(-\eta)|\le\left\{\!\begin{alignedat}{2}
& \f{C_0}{1+\eta}, & \quad & \forall\,\be\in[1,2],\\
& \f{C_0}{1+\eta^2}, & \quad & \be=\al
\end{alignedat}\right.\quad\mbox{for }\eta\ge0,\label{eq-est-ML0}\\
\left|E_{\al,j}(-\eta)-\f1{\Ga(j-\al)\,\eta}\right|\le\f{C_0'}{\eta^2}\quad\mbox{for }\eta\gg1,\ j=1,2.\label{eq-est-ML1}
\end{gather}
\end{lemma}

For $\be\ne\al$, the estimate \eqref{eq-est-ML0} follows immediately from \cite[Theorem 1.6]{Po}. Only in the special case of $\be=\al$, one can apply \cite[Theorem 1.4]{Po} with $p=1$ to improve the estimate. Similarly, the estimate \eqref{eq-est-ML1} also follows from the asymptotic estimate of $E_{\al,j}(z)$ with $j=1,2$ in \cite[Theorem 1.4]{Po}.

Now we proceed to the space direction. By $(\,\cdot\,,\,\cdot\,)$ we denote the usual inner product in $L^2(\Om)$, and let $H^\ga(\Om)$ ($\ga>0$) denote Sobolev spaces (e.g. Adams \cite{Ad}). The elliptic operator $\cA$ in \eqref{eq-ibvp-u0}--\eqref{eq-ibvp-u1} is defined by
\[
\cA:H^2(\Om)\cap H_0^1(\Om)\longrightarrow L^2(\Om),\quad g\longmapsto-\nb\cdot(\bm a\nb g)+c\,g,
\]
where $\cdot$ and $\nb$ refer to the inner product in $\BR^d$ and the gradient in $\bm x$, respectively. Here $c\in L^\infty(\Om)$ is non-negative and $\bm a=(a_{i j})_{1\le i,j\le d}\in C^1(\ov\Om;\BR_{\mathrm{sym}}^{d\times d})$ is a symmetric and strictly positive-definite matrix-valued function on $\ov\Om$. More precisely, we assume that
\[
c\ge0\mbox{ in }\Om,\quad a_{i j}\in C^1(\Om),\quad a_{i j}=a_{j i}\mbox{ on }\ov\Om\quad(1\le i,j\le d)
\]
and there exists a constant $\ka>0$ such that
\[
\bm a(\bm x)\bm\xi\cdot\bm\xi\ge\ka(\bm\xi\cdot\bm\xi),\quad\forall\,\bm x\in\ov\Om,\ \forall\,\bm\xi\in\BR^d.
\]
Next, we introduce the eigensystem $\{(\la_n,\vp_n)\}_{n=1}^\infty$ of $\cA$ satisfying
\[
\cA\vp_n=\la_n\vp_n,\quad0<\la_1<\la_2\le\cdots,\quad\la_n\longrightarrow\infty\ (n\to\infty)
\]
and $\{\vp_n\}\subset\cD(\cA)$ forms a complete orthonormal system of $L^2(\Om)$. As usual, we can further introduce the Hilbert space $\cD(\cA^\be)$ for $\be\ge0$ as
\[
\cD(\cA^\be):=\left\{g\in L^2(\Om)\left|\|g\|_{\cD(\cA^\be)}:=\left(\sum_{n=1}^\infty\left|\la_n^\be(g,\vp_n)\right|^2\right)^{1/2}<\infty\right.\right\}.
\]
We know $\cD(\cA^\be)\subset H^{2\be}(\Om)$ for $\be>0$. For $-1\le\be<0$, let
\[
\cD(\cA^{-\be})\subset L^2(\Om)\subset(\cD(\cA^{-\be}))'=:\cD(\cA^\be)
\]
be the Gel'fand triple, where $(\,\cdot\,)'$ denotes the dual space. Then the norm of $\cD(\cA^\be)$ for $-1\le\be<0$ is similarly defined by
\[
\|g\|_{\cD(\cA^\be)}:=\left(\sum_{n=1}^\infty\left|\la_n^\be\,{}_{\cD(\cA^\be)}\langle g,\vp_n\rangle_{\cD(\cA^{-\be})}\right|^2\right)^{1/2},
\]
where ${}_{\cD(\cA^\be)}\langle\,\cdot\,,\,\cdot\,\rangle_{\cD(\cA^{-\be})}$ denotes the pairing between $\cD(\cA^\be)$ and $\cD(\cA^{-\be})$. Then the space $\cD(\cA^\be)$ is well-defined for all $\be\ge-1$.

Now we are well prepared to state the main results of this article. First we investigate the asymptotic behavior of the solution to the homogeneous problem \eqref{eq-ibvp-u0} as $t\to\infty$.

\begin{theorem}[Long-time asymptotic estimate]\label{thm-asymp}
Let $u_j\in\cD(\cA^\be)$ $(j=0,1)$ with $\be\ge0$ and $u$ be the solution to \eqref{eq-ibvp-u0}. Then there exists a constant $C_\al>0$ depending only on $\al$ such that
\begin{equation}\label{eq-asymp}
\left\|u(\,\cdot\,,t)-\sum_{j=0}^1\f{\cA^{-1}u_j}{\Ga(j+1-\al)}t^{j-\al}\right\|_{\cD(\cA^{\be+1})}\le C_\al\sum_{j=0}^1\|u_j\|_{\cD(\cA^{\be-1})}t^{j-2\al}
\end{equation}
for $t\gg1$.
\end{theorem}

The above theorem generalizes a similar result for multi-term time-fractional diffusion equations in Li, Liu and Yamamoto \cite[Theorem 2.4]{LLY15}, and the estimate \eqref{eq-asymp} keeps the same structure describing the asymptotic behavior. First, \eqref{eq-asymp} points out that the solution $u(\,\cdot\,,t)$ converges to $0$ in $\cD(\cA^{\be+1})$ with the pattern
\[
\sum_{j=0}^1\f{\cA^{-1}u_j}{\Ga(j+1-\al)}t^{j-\al}=\f{\cA^{-1}u_0}{\Ga(1-\al)}t^{-\al}+\f{\cA^{-1}u_1}{\Ga(2-\al)}t^{1-\al}\quad\mbox{as }t\to\infty,
\]
which immediately implies
\[
\|u(\,\cdot\,,t)\|_{\cD(\cA^{\be+1})}=\begin{cases}
O(t^{1-\al}), & u_1\not\equiv0\mbox{ in }\Om,\\
O(t^{-\al}), & u_1\equiv0,\ u_0\not\equiv0\mbox{ in }\Om
\end{cases}\quad\mbox{as }t\to\infty.
\]
Therefore, the initial velocity $u_1$ impacts the long-time asymptotic behavior more than the initial displacement $u_0$. Second, \eqref{eq-asymp} further gives the convergence rate
\[
\left\|u(\,\cdot\,,t)-\sum_{j=0}^1\f{\cA^{-1}u_j}{\Ga(j+1-\al)}t^{j-\al}\right\|_{\cD(\cA^{\be+1})}=\begin{cases}
O(t^{1-2\al}), & u_1\not\equiv0\mbox{ in }\Om,\\
O(t^{-2\al}), & u_1\equiv0,\ u_0\not\equiv0\mbox{ in }\Om
\end{cases}
\]
for $t\gg1$. In this sense, the estimate \eqref{eq-asymp} provides rich information on the long-time asymptotic behavior of the solution.

As a direct consequence of Theorem~\ref{thm-asymp}, one can immediately show the following result on the sign of the solution for $t\gg1$.

\begin{corollary}[Long-time strict positivity/negativity]\label{cor-asymp}
Let $u_j\in\cD(\cA^\be)$ $(j=0,1)$ with
\begin{equation}\label{eq-cond-be}
\be\begin{cases}
=0, & d=1,2,3,\\
>d/4-1, & d\ge4
\end{cases}
\end{equation}
and $u$ be the solution to \eqref{eq-ibvp-u0}. Then for any $\bm x\in\Om,$ the followings hold true.

{\rm(a)} If $\cA^{-1}u_1(\bm x)=0$ and $\cA^{-1}u_0(\bm x)\ne0,$ then there exists a constant $T_0\gg1$ depending on $\al,\Om,\cA,u_0,u_1,\bm x$ such that the sign of $u(\bm x,\,\cdot\,)$ is opposite to that of $\cA^{-1}u_0(\bm x)$ in $(T_0,\infty)$. Especially$,$ if $u_0\not\equiv0$ and $u_0\le0$ $(u_0\ge0)$ in $\Om,$ then $u(\bm x,\,\cdot\,)>0$ $(u(\bm x,\,\cdot\,)<0)$ in $(T_0,\infty)$.

{\rm(b)} If $\cA^{-1}u_1(\bm x)\ne0,$ then there exists a constant $T_1\gg1$ depending on $\al,\Om,\cA,u_0,u_1,\bm x$ such that the sign of $u(\bm x,\,\cdot\,)$ is the same as that of $\cA^{-1}u_1(\bm x)$ in $(T_1,\infty)$. Especially$,$ if $u_1\not\equiv0$ and $u_1\ge0$ $(u_1\le0)$ in $\Om,$ then $u(\bm x,\,\cdot\,)>0$ $(u(\bm x,\,\cdot\,)<0)$ in $(T_1,\infty)$.
\end{corollary}

Similarly to Theorem~\ref{thm-asymp}, the above corollary also generalizes a similar result for multi-term time-fractional diffusion equations (see Liu \cite[Lemma 3.1]{L17}). However, since solutions to time-fractional wave equations change sign in general, there seems no literature discussing the sign of solutions for $1<\al<2$ to our best knowledge. On the other hand, notice that in Corollary~\ref{cor-asymp} we directly make assumptions on $\cA^{-1}u_j(\bm x)$ ($j=0,1$) instead of $u_j$ as that in \cite{L17}. Since the non-vanishing of $\cA^{-1}u_j(\bm x)$ is a necessary condition of
\begin{equation}\label{eq-SMP}
u_j\not\equiv0\quad\mbox{and}\quad(u_j\ge0\mbox{ or }u_j\le0)\quad\mbox{in }\Om
\end{equation}
according to the strong maximum principle, the assumptions in Corollary~\ref{cor-asymp} are definitely weaker than the previous one.

Under the special situation \eqref{eq-SMP}, let us comment Corollary~\ref{cor-asymp} in further detail. If the initial velocity $u_1$ vanishes and the sign of the initial displacement $u_0(\not\equiv0)$ keeps unchanged, then Corollary~\ref{cor-asymp}(a) asserts that the solution $u$ to the homogeneous problem \eqref{eq-ibvp-u0} must take the opposite sign against that of $u_0$ for $t\gg1$. This turns out to be the remarkable difference from the case of $0<\al\le1$ in view of the strong positivity property for the latter.

In contrast, if the sign of $u_1(\not\equiv0)$ keeps unchanged, then Corollary~\ref{cor-asymp}(b) claims that $u$ must takes the same sign as that of $u_1$ for $t\gg1$. Notice that in Corollary~\ref{cor-asymp}(b), there is no assumption on $u_0$ because $u_1$ plays a more dominating role in the asymptotic estimate \eqref{eq-asymp} than $u_0$ does.

Corollary~\ref{cor-asymp} is not only novel and interesting by itself, but also closely related to the uniqueness issue of Problem \ref{isp}. Indeed, as a direct application of Corollary~\ref{cor-asymp}, one can prove the following theorem.

\begin{theorem}[Uniqueness for Problem \ref{isp}]\label{thm-isp}
Let $\bm x_0\in\Om$ and $u$ be the solution to $\eqref{eq-ibvp-u1},$ where $\rho\in L^p(0,T)$ $(1\le p\le\infty)$ and $f\in\cD(\cA^\be)$ with $\be$ satisfying \eqref{eq-cond-be}. If $\cA^{-1}f(\bm x_0)\ne0,$ then $u=0$ at $\{\bm x_0\}\times(0,T)$ implies $\rho\equiv0$ in $(0,T)$. Especially$,$ if $f\not\equiv0$ and $(f\ge0\mbox{ or }f\le0)$ in $\Om,$ then the same result holds with arbitrary $\bm x_0\in\Om$.
\end{theorem}

As expected, again the above theorem generalizes and improves corresponding results for (multi-term) time-fractional diffusion equations (see Liu, Rundell and Yamamoto \cite[Theorem 1.2]{LRY} and Liu \cite[Theorem 1.3]{L17}). Moreover, in spite of the difference between time-fractional diffusion and wave equations, the assumption and result keep essentially the same. Therefore, the key ingredient for proving such uniqueness turns out to be the long-time strict positivity/negativity like Corollary~\ref{cor-asymp} instead of the strong positivity property. Further, thanks to the weakened assumption in Corollary~\ref{cor-asymp}(b), here in Theorem~\ref{thm-isp} we can also weaken the sign assumption on $f$ simply to $\cA^{-1}f(\bm x_0)\ne0$.


\section{Well-Posedness and Fractional Duhamel's Principle}\label{sec-forward}

This section is devoted to the preparation of basic facts concerning problems \eqref{eq-ibvp-u0}--\eqref{eq-ibvp-u1} before proceeding to the proofs of main results.

We start with discussing the unique existence and regularity of solutions to \eqref{eq-ibvp-u0}--\eqref{eq-ibvp-u1}. Regarding the well-posedness of time-fractional wave equations, there are partial results e.g. in \cite{LHY21,SY} and here we generalize their results to fit into the framework of this paper. First we consider the homogeneous problem \eqref{eq-ibvp-u0}.

\begin{lemma}\label{lem-reg1}
Fix constants $\be\ge0,$ $\ga\in[0,1]$ arbitrarily and assume $u_j\in\cD(\cA^\be)$ $(j=0,1)$. Then there exists a unique solution $u\in L^\infty(0,T;\cD(\cA^\be))$ to the initial-boundary value problem \eqref{eq-ibvp-u0}. Moreover$,$ there exists a constant $C_1>0$ depending only on $\al$ such that
\begin{equation}\label{eq-est1}
\|u(\,\cdot\,,t)\|_{\cD(\cA^{\be+\ga})}\le C_1\sum_{j=0}^1\|u_j\|_{\cD(\cA^\be)}t^{j-\al\ga},\quad t>0.
\end{equation}
Further$,$ the map $u:(0,T)\longrightarrow\cD(\cA^{\be+1})$ can be analytically extended to a sector $\{z\in\BC\setminus\{0\}\mid|\mathrm{arg}\,z|<\pi/2\}$.
\end{lemma}

\begin{proof}
It follows from \cite{SY} that the solution to \eqref{eq-ibvp-u0} takes the form
\begin{equation}\label{eq-sol0}
u(\,\cdot\,,t)=\sum_{j=0}^1t^j\sum_{n=1}^\infty E_{\al,j+1}(-\la_n t^\al)(u_j,\vp_n)\vp_n.
\end{equation}
Employing the estimate \eqref{eq-est-ML0} in Lemma~\ref{lem-asymp}, we estimate
\begin{align*}
\|u(\,\cdot\,,t)\|_{\cD(\cA^{\be+\ga})}^2 & =\sum_{n=1}^\infty\la_n^{2(\be+\ga)}\left|\sum_{j=0}^1t^j E_{\al,j+1}(-\la_n t^\al)(u_j,\vp_n)\right|^2\\
& \le2\sum_{j=0}^1t^{2j}\sum_{n=1}^\infty|\la_n^\ga E_{\al,j+1}(-\la_n t^\al)|^2\left|\la_n^\be(u_j,\vp_n)\right|^2\\
& \le2\sum_{j=0}^1t^{2j}\sum_{n=1}^\infty\left(\f{C_0(\la_n t^\al)^\ga}{1+\la_n t^\al}t^{-\al\ga}\right)^2\left|\la_n^\be(u_j,\vp_n)\right|^2\\
& \le2\left(C_0\,t^{-\al\ga}\right)^2\sum_{j=0}^1\left(\|u_j\|_{\cD(\cA^\be)}t^j\right)^2\\
& \le2\left(C_0\sum_{j=0}^1\|u_j\|_{\cD(\cA^\be)}t^{j-\al\ga}\right)^2.
\end{align*}
Then we arrive at \eqref{eq-est1} by simply putting $C_1=\sqrt2\,C_0$. In particular, taking $\ga=0$ in \eqref{eq-est1} immediately yields $u\in L^\infty(0,T;\cD(\cA^\be))$.

Finally, the analyticity of $u(\,\cdot\,,t)$ with respect to $t$ in $\cD(\cA^{\be+1})$ can be proved by the same argument as that of \cite[Theorem 2.1]{SY}.\qed
\end{proof}

Next, we consider the inhomogeneous problem with a general source term:
\begin{equation}\label{eq-ibvp-w}
\begin{cases}
(\pa_t^\al+\cA)w=F & \mbox{in }\Om\times(0,T),\\
w=\pa_t w=0 & \mbox{in }\Om\times\{0\},\\
w=0 & \mbox{on }\pa\Om\times(0,T).
\end{cases}
\end{equation}

\begin{lemma}\label{lem-reg2}
Fix constants $\be\ge0,$ $p\in[1,\infty]$ arbitrarily and assume $F\in L^p(0,T;$ $\cD(\cA^\be))$.

{\rm(a)} If $p=2,$ then there exist a unique solution $w\in L^p(0,T;$ $\cD(\cA^{\be+1}))$ to \eqref{eq-ibvp-w} and a constant $C_2>0$ depending only on $\al,T$ such that
\[
\|w\|_{L^p(0,T;\cD(\cA^{\be+1}))}\le C_2\|F\|_{L^p(0,T;\cD(\cA^\be))}.
\]

{\rm(b)} If $p\ne2,$ then there exist a unique solution $w\in L^p(0,T;$ $\cD(\cA^{\be+\ga}))$ to \eqref{eq-ibvp-w} for any $\ga\in[0,1)$ and a constant $C_2'>0$ depending only on $\al,T$ such that
\[
\|w\|_{L^p(0,T;\cD(\cA^{\be+\ga}))}\le\f{C_2'}{1-\ga}\|F\|_{L^p(0,T;\cD(\cA^\be))}.
\]
\end{lemma}

The results of the above lemma inherits those of \cite[Theorem 2.2(b)]{LLY15}, \cite[Lemma 2.3(a)]{LHY21} and \cite[Theorem 5(ii)]{LHL22}. More precisely, the improvement of the spatial regularity of the solution can reach $2$ only for $p=2$, which is strictly smaller than $2$ if $p\ne2$. The proof of Lemma~\ref{lem-reg2}(b) resembles those in the above references and we omit the proof. However, recall that for $p=2$, the proof e.g. in \cite{LLY15} relies on the positivity of $E_{\al,\al}(-\eta)$ for $\al\in(0,1)$ and $\eta\ge0$. Since such positivity no longer holds for $\al\in(1,2)$, we shall give an alternative proof for Lemma~\ref{lem-reg2}(a).

\begin{proof}[{Proof of Lemma~\ref{lem-reg2}(a)}]
Based on the solution formula (see \cite[Theorem 2.2]{SY})
\begin{equation}\label{eq-sol1}
w(\,\cdot\,,t)=\sum_{n=1}^\infty\left(\int_0^t(t-s)^{\al-1}E_{\al,\al}(-\la_n(t-s)^\al)(F(\,\cdot\,,s),\vp_n)\,\rd s\right)\vp_n,
\end{equation}
we employ Young's convolution inequality to estimate
\begin{align*}
& \quad\,\|w\|_{L^2(0,T;\cD(\cA^{\be+1}))}^2=\int_0^T\|w(\,\cdot\,,t)\|_{\cD(\cA^{\be+1})}^2\,\rd t\\
& =\int_0^T\sum_{n=1}^\infty\left|\la_n^{\be+1}\int_0^t(t-s)^{\al-1}E_{\al,\al}(-\la_n(t-s)^\al)(F(\,\cdot\,,s),\vp_n)\,\rd s\right|^2\rd t\\
& =\sum_{n=1}^\infty\int_0^T\left|\int_0^t\la_n^{\be+1}(t-s)^{\al-1}E_{\al,\al}(-\la_n(t-s)^\al)(F(\,\cdot\,,s),\vp_n)\,\rd s\right|^2\rd t\\
& \le\sum_{n=1}^\infty\left(\int_0^T\la_n t^{\al-1}|E_{\al,\al}(-\la_n t^\al)|\,\rd t\right)^2\int_0^T\left|\la_n^\be(F(\,\cdot\,,t),\vp_n)\right|^2\rd t.
\end{align*}
It suffices to show the uniform boundedness of $\int_0^T\la_n t^{\al-1}|E_{\al,\al}(-\la_n t^\al)|\,\rd t$ for all $n=1,2,\ldots$. Indeed, performing integration by substitution $\eta=\la_n t^\al$ and utilizing the estimate \eqref{eq-est-ML0} with $\be=\al$ in Lemma~\ref{lem-asymp}, we calculate
\begin{align*}
\int_0^T\la_n t^{\al-1}|E_{\al,\al}(-\la_n t^\al)|\,\rd t & =\f1\al\int_0^{\la_n T^\al}|E_{\al,\al}(-\eta)|\,\rd\eta\le\f1\al\int_0^\infty|E_{\al,\al}(-\eta)|\,\rd\eta\\
& \le\f{C_0}\al\int_0^\infty\f{\rd\eta}{1+\eta^2}=\f{\pi C_0}{2\al}=:C_2.
\end{align*}
Then we immediately obtain
\[
\|u\|_{L^2(0,T;\cD(\cA^{\be+1}))}^2\le C_2^2\sum_{n=1}^\infty\int_0^T\left|\la_n^\be(F(\,\cdot\,,t),\vp_n)\right|^2\rd t=\left(C_2\|F\|_{L^2(0,T;\cD(\cA^\be))}\right)^2,
\]
which completes the proof of Lemma~\ref{lem-reg2}(a).\qed
\end{proof}

Now we discuss the fractional Duhamel's principle which connects the inhomogeneous problem \eqref{eq-ibvp-u1} and the homogeneous one \eqref{eq-ibvp-u0}. Concerning Duhamel's principle for time-fractional partial differential equations, there already exist plentiful results especially for time-fractional diffusion equations, and we refer e.g. to \cite{LRY,L17,Um}. For general $\al>0$, Hu, Liu and Yamamoto \cite[Lemma 5.2]{HLY20} established a fractional Duhamel's principle for \eqref{eq-ibvp-w} with a smooth source term $F$. Here we provide a result for \eqref{eq-ibvp-u1} with a non-smooth source term.

\begin{lemma}[Fractional Duhamel's principle]\label{lem-Duhamel}
Fix constants $\be\ge0,$ $p\in[1,\infty]$ arbitrarily and assume $\rho\in L^p(0,T),$ $f\in\cD(\cA^\be)$. Let $\ga=1$ for $p=2$ and $\ga\in[0,1)$ be arbitrary for $p\ne2$. Then for the solution $u$ to $\eqref{eq-ibvp-u1},$ there holds
\begin{equation}\label{eq-Duhamel}
J^{2-\al}u(\,\cdot\,,t)=\int_0^t\rho(s)v(\,\cdot\,,t-s)\,\rd s\quad\mbox{in }L^p(0,T;\cD(\cA^{\be+\ga})),
\end{equation}
where $v$ solves the homogeneous problem
\begin{equation}\label{eq-ibvp-v}
\begin{cases}
(\pa_t^\al+\cA)v=0 & \mbox{in }\Om\times(0,T),\\
v=0,\ \pa_t v=f & \mbox{in }\Om\times\{0\},\\
v=0 & \mbox{on }\pa\Om\times(0,T).
\end{cases}
\end{equation}
\end{lemma}

\begin{proof}
First we confirm that the both sides of \eqref{eq-Duhamel} lie in $L^p(0,T;\cD(\cA^{\be+\ga}))$ for $\ga$ assumed in Lemma~\ref{lem-Duhamel}. In fact, since $\rho f\in L^p(0,T;\cD(\cA^\be))$, it follows from Lemma~\ref{lem-reg2} that $u\in L^p(0,T;\cD(\cA^{\be+\ga}))$. Next, it is readily seen from Young's convolution inequality that $J^{2-\al}:L^p(0,T)\longrightarrow L^p(0,T)$ is a bounded linear operator, which implies $J^{2-\al}u\in L^p(0,T;\cD(\cA^{\be+\ga}))$.

On the other hand, applying Lemma~\ref{lem-reg1} to \eqref{eq-ibvp-v} yields
\[
\|v(\,\cdot\,,t)\|_{\cD(\cA^{\be+\ga})}\le C_1\|f\|_{\cD(\cA^\be)}t^{1-\al\ga},\quad t>0,\ \forall\,\ga\in[0,1].
\]
Since $1-\al\ga>-1$, we have $v\in L^1(0,T;\cD(\cA^{\be+\ga}))$ for any $\ga\in[0,1]$. Therefore, we see that the right-hand side of \eqref{eq-Duhamel} also makes sense in $L^p(0,T;\cD(\cA^{\be+\ga}))$ for any $\ga\in[0,1]$ by $\rho\in L^p(0,T)$ and again Young's convolution inequality.

Now we can proceed to verify the identity \eqref{eq-Duhamel} by brute-force calculation based on the solution formulae, because the possibility of exchanging the involved summations and integrals is guaranteed by the above argument. According to \eqref{eq-sol1}, we know
\begin{align}
u(\,\cdot\,,t) & =\sum_{n=1}^\infty\left(\int_0^t(t-s)^{\al-1}E_{\al,\al}(-\la_n(t-s)^\al)(\rho(s)f,\vp_n)\,\rd s\right)\vp_n\nonumber\\
& =\sum_{n=1}^\infty\mu_n(t)(f,\vp_n)\vp_n,\label{eq-rep}
\end{align}
where
\[
\mu_n(t):=\int_0^t(t-s)^{\al-1}E_{\al,\al}(-\la_n(t-s)^\al)\rho(s)\,\rd s.
\]
Then by the definitions of $J^{2-\al}$ and $E_{\al,\al}(\,\cdot\,)$, we calculate
\begin{align}
& \quad\,J^{2-\al}\mu_n(t)=\f1{\Ga(2-\al)}\int_0^t(t-s)^{1-\al}\mu_n(s)\,\rd s\nonumber\\
& =\f1{\Ga(2-\al)}\int_0^t(t-s)^{1-\al}\left(\int_0^s(s-\tau)^{\al-1}E_{\al,\al}(-\la_n(s-\tau)^\al)\rho(\tau)\,\rd\tau\right)\rd s\label{eq-exchange}\\
& =\f1{\Ga(2-\al)}\int_0^t\rho(\tau)\left(\int_\tau^t(t-s)^{1-\al}(s-\tau)^{\al-1}\sum_{k=0}^\infty\f{(-\la_n(s-\tau)^\al)^k}{\Ga(\al k+\al)}\,\rd s\right)\rd\tau\nonumber\\
& =\f1{\Ga(2-\al)}\int_0^t\rho(\tau)\left(\sum_{k=0}^\infty\f{(-\la_n)^k}{\Ga(\al(k+1))}\int_\tau^t(t-s)^{1-\al}(s-\tau)^{\al(k+1)-1}\,\rd s\right)\rd\tau,\nonumber
\end{align}
where we exchanged the order of integration in \eqref{eq-exchange}. For the inner integral above, we perform integration by substitution $s=\te(t-\tau)+\tau$ ($0<\te<1$) to calculate
\begin{align*}
\int_\tau^t(t-s)^{1-\al}(s-\tau)^{\al(k+1)-1}\,\rd s & =(t-\tau)^{\al k+1}\int_0^1(1-\te)^{1-\al}\te^{\al(k+1)-1}\,\rd\te\\
& =(t-\tau)^{\al k+1}\f{\Ga(2-\al)\Ga(\al(k+1))}{\Ga(\al k+2)}.
\end{align*}
Hence, by the definition of $E_{\al,2}(\,\cdot\,)$, we obtain
\begin{align*}
J^{2-\al}\mu_n(t) & =\int_0^t\rho(\tau)\sum_{k=0}^\infty\f{(-\la_n)^k}{\Ga(\al k+2)}(t-\tau)^{\al k+1}\,\rd\tau\\
& =\int_0^t\rho(s)(t-s)E_{\al,2}(-\la_n(t-s)^\al)\,\rd s.
\end{align*}
Therefore, we perform $J^{2-\al}$ on both sides of \eqref{eq-rep} and substitute the above equality to obtain
\begin{align*}
J^{2-\al}u(\,\cdot\,,t) & =\sum_{n=1}^\infty J^{2-\al}\mu_n(t)(f,\vp_n)\vp_n\\
& =\sum_{n=1}^\infty\left(\int_0^t\rho(s)(t-s)E_{\al,2}(-\la_n(t-s)^\al)\,\rd s\right)(f,\vp_n)\vp_n\\
& =\int_0^t\rho(s)\left((t-s)\sum_{n=1}^\infty E_{\al,2}(-\la_n(t-s)^\al)(f,\vp_n)\vp_n\right)\rd s.
\end{align*}
Then we reach the conclusion by applying the solution formula \eqref{eq-sol0} to $v$.\qed
\end{proof}


\section{Proofs of Main Results}\label{sec-proof}

\begin{proof}[Proof of Theorem~\ref{thm-asymp}]
Since $u(\,\cdot\,,t)\in\cD(\cA^{\be+1})$ by Lemma~\ref{lem-reg1} with $\ga=1$, we know that the left-hand side of \eqref{eq-asymp} makes sense for $t>0$. By the solution formula \eqref{eq-sol0} and
\[
\cA^{-1}u_j=\sum_{n=1}^\infty\f{(u_j,\vp_n)}{\la_n}\vp_n,\quad j=0,1,
\]
we obtain
\begin{align*}
& \quad\,u(\,\cdot\,,t)-\sum_{j=0}^1t^{j-\al}\f{\cA^{-1}u_j}{\Ga(j+1-\al)}\\
& =\sum_{j=0}^1t^j\sum_{n=1}^\infty\left(E_{\al,j+1}(-\la_n t^\al)-\f1{\Ga(j+1-\al)\,\la_n t^\al}\right)(u_j,\vp_n)\vp_n.
\end{align*}
Then we employ the estimate \eqref{eq-est-ML1} in Lemma~\ref{lem-asymp} to estimate
\begin{align*}
& \quad\,\left\|u(\,\cdot\,,t)-\sum_{j=0}^1\f{\cA^{-1}u_j}{\Ga(j+1-\al)}t^{j-\al}\right\|_{\cD(\cA^{\be+1})}^2\\
& =\sum_{n=1}^\infty\la_n^{2(\be+1)}\left|\sum_{j=0}^1t^j\left(E_{\al,j+1}(-\la_n t^\al)-\f1{\Ga(j+1-\al)\,\la_n t^\al}\right)(u_j,\vp_n)\right|^2\\
& \le2\sum_{j=0}^1t^{2j}\sum_{n=1}^\infty\la_n^2\left|E_{\al,j+1}(-\la_n t^\al)-\f1{\Ga(j+1-\al)\,\la_n t^\al}\right|^2\left|\la_n^\be(u_j,\vp_n)\right|^2\\
& \le2\left(C_0'\,t^{-2\al}\right)^2\sum_{j=0}^1t^{2j}\sum_{n=1}^\infty\left|\la_n^{\be-1}(u_j,\vp_n)\right|^2\le2\left(C_0'\sum_{j=0}^1\|u_j\|_{\cD(\cA^{\be-1})}t^{j-2\al}\right)^2
\end{align*}
for $t\gg1$. Then we can conclude \eqref{eq-asymp} by simply putting $C_\al=\sqrt2\,C_0'$.\qed
\end{proof}

\begin{proof}[Proof of Corollary~\ref{cor-asymp}]
We have from \eqref{eq-cond-be} that $2(\be+1)>d/2$ and the Sobolev embedding yields $\cD(\cA^{\be+1})\subset H^{2(\be+1)}\subset C(\ov\Om)$ (e.g. Adams \cite{Ad}), indicating that $\cA^{-1}u_j$ ($j=0,1$) and $u(\,\cdot\,,t)$ ($t>0$) make pointwise sense. Then Theorem~\ref{thm-asymp} yields for any $\bm x\in\Om$ that
\begin{align*}
& \quad\,\left|u(\bm x,t)-\sum_{j=0}^1\f{\cA^{-1}u_j(\bm x)}{\Ga(j+1-\al)}t^{j-\al}\right|\le\left\|u(\,\cdot\,,t)-\sum_{j=0}^1\f{\cA^{-1}u_j}{\Ga(j+1-\al)}t^{j-\al}\right\|_{C(\ov\Om)}\\
& \le C_\Om\left\|u(\,\cdot\,,t)-\sum_{j=0}^1\f{\cA^{-1}u_j}{\Ga(j+1-\al)}t^{j-\al}\right\|_{\cD(\cA^{\be+1})}\le C\sum_{j=0}^1\|u_j\|_{\cD(\cA^{\be-1})}t^{j-2\al}
\end{align*}
for $t\gg1$, where $C_\Om>0$ is the Sobolev embedding constant depending only on $\Om$ and $C:=C_\Om C_\al$ depends only on $\Om$ and $\al$. This implies
\begin{align}
u(\bm x,t) & \ge\sum_{j=0}^1\f{\cA^{-1}u_j(\bm x)}{\Ga(j+1-\al)}t^{j-\al}-C\sum_{j=0}^1\|u_j\|_{\cD(\cA^{\be-1})}t^{j-2\al},\label{eq-pos}\\
u(\bm x,t) & \le\sum_{j=0}^1\f{\cA^{-1}u_j(\bm x)}{\Ga(j+1-\al)}t^{j-\al}+C\sum_{j=0}^1\|u_j\|_{\cD(\cA^{\be-1})}t^{j-2\al}\nonumber
\end{align}
for $t\gg1$. Recall the relation $0>1-\al>-\al>1-2\al>-2\al$ by $1<\al<2$.\medskip

(a) Without loss of generality, we only consider the case of $\cA^{-1}u_0(\bm x)<0$ because the opposite case can be studied in identically the same manner.

Substituting $\cA^{-1}u_1(\bm x)=0$ into \eqref{eq-pos} yields
\begin{equation}\label{eq-pos0}
u(\bm x,t)\ge\f{\cA^{-1}u_0(\bm x)}{\Ga(1-\al)}t^{-\al}-C\sum_{j=0}^1\|u_j\|_{\cD(\cA^{\be-1})}t^{j-2\al}\quad\mbox{for }t\gg1.
\end{equation}
Together with the fact that $\Ga(1-\al)<0$, we see $\f{\cA^{-1}u_0(\bm x)}{\Ga(1-\al)}>0$ and thus there exists $T_0\gg1$ depending on $\al,\Om,\cA,u_0,u_1,\bm x$ such that the right-hand side of \eqref{eq-pos0} keeps strictly positive for all $t>T_0$.

As for the special situation e.g. of $u_0\not\equiv0,\le0$ in $\Om$, the strong maximum principle for the elliptic operator $\cA$ (see Gilbarg and Trudinger \cite[Chapter 3]{GT}) guarantees $\cA^{-1}u_0<0$ in $\Om$. This completes the proof of (a).\medskip

(b) Likewise, it suffices to consider the case of $\cA^{-1}u_1(\bm x)>0$ without loss of generality. Now it follows from \eqref{eq-pos} that
\begin{equation}\label{eq-pos1}
u(\bm x,t)\ge\f{\cA^{-1}u_1(\bm x)}{\Ga(2-\al)}t^{1-\al}-\left|\f{\cA^{-1}u_0(\bm x)}{\Ga(1-\al)}\right|t^{-\al}-C\sum_{j=0}^1\|u_j\|_{\cD(\cA^{\be-1})}t^{j-2\al}
\end{equation}
for $t\gg1$. Then by $\Ga(2-\al)>0$, similarly we conclude the existence of $T_1\gg1$ depending on $\al,\Om,\cA,u_0,u_1,\bm x$ such that the right-hand side of \eqref{eq-pos1} keeps strictly positive for all $t> T_1$. This completes the proof of (b).\qed
\end{proof}

\begin{proof}[Proof of Theorem~\ref{thm-isp}]
We turn to the fractional Duhamel's principle \eqref{eq-Duhamel} and the corresponding homogeneous problem \eqref{eq-ibvp-v} in Lemma~\ref{lem-Duhamel}. Since the identity \eqref{eq-Duhamel} holds in $L^p(0,T;\cD(\cA^{\be+\ga}))$ for any $\ga\in[0,1)$ and $\be$ satisfies \eqref{eq-cond-be}, one can choose $\ga$ sufficiently close to $1$ such that $2(\be+\ga)>d/2$. Then by the Sobolev embedding $\cD(\cA^{\be+\ga})\subset H^{2(\be+\ga)}(\Om)\subset C(\ov\Om)$, it reveals that \eqref{eq-Duhamel} makes sense in $L^p(0,T;C(\ov\Om))$. This allows us to substitute $\bm x=\bm x_0$ into \eqref{eq-Duhamel} to obtain
\[
0=J^{2-\al}u(\bm x_0,t)=\int_0^t\rho(s)v(\bm x_0,t-s)\,\rd s\quad\mbox{in }L^p(0,T).
\]

Now we are well prepared to apply the Titchmarsh convolution theorem (see \cite{Ti}) to conclude the existence of a constant $t_*\in[0,T]$ such that
\[
v(\bm x_0,\,\cdot\,)\equiv0\mbox{ in }(0,t_*)\quad\mbox{and}\quad\rho\equiv0\mbox{ in }(0,T-t_*).
\]
It remains to verify $t_*=0$ by the argument of contradiction. If $t_*>0$ instead, then the analyticity of $v(\,\cdot\,,t)\in\cD(\cA^{\be+1})\subset C(\ov\Om)$ with respect to $t$ indicates $v(\bm x_0,\,\cdot\,)\equiv0$ in $(0,\infty)$. However, owing to the assumption $\cA^{-1}f(\bm x_0)\ne0$, it follows from Corollary~\ref{cor-asymp}(b) that there exists $T_1\gg1$ such that
\[
v(\bm x_0,\,\cdot\,)>0\quad\mbox{or}\quad v(\bm x_0,\,\cdot\,)<0\quad\mbox{in }(T_1,\infty),
\]
which is a contradiction. Consequently, there should hold $t_*=0$, i.e., $\rho\equiv0$ in $(0,T)$.\qed
\end{proof}


\section{Conclusion}\label{sec-remark}

The results obtained in this article reflect both similarity and difference between time-fractional wave and diffusion equations. Owing to the existence of two initial values, the evolution of solutions $u$ to homogeneous time-fractional wave equations becomes more complicated and interesting than that for $0<\al<1$. Motivated by the asymptotic property of Mittag-Leffler functions, we capture the behavior of $u(\,\cdot\,,t)$ for $t\gg1$ in view of Theorem~\ref{thm-asymp} and Corollary~\ref{cor-asymp}. Meanwhile, it reveals that the uniqueness of inverse $t$-source problems like Problem~\ref{isp} only requires the long-time non-vanishing and the time-analyticity of $u$, which was neglected in the proof for $0<\al<1$.

We close this article with some conjectures on the number of sign changes of the solution $u(\bm x,t)$ to \eqref{eq-ibvp-u0} for any $\bm x\in\Om$. Fixing initial values $u_0,u_1$, from Figure~\ref{fig-ML} one can see the monotone increasing of such a number with respect to $\al$, but there seems no rigorous proof yet. Further, if either of $u_0,u_1$ vanishes and the other keeps sign, it seems possible to prove the parity of the number of sign changes, i.e., odd if $u_1\equiv0$ and even if $u_0\equiv0$. This also relates with the distribution of zeros of $u$, which can be another future topic.\bigskip


{\bf Acknowledgement}\ \ 
X.\! Huang is supported by Grant-in-Aid for JSPS Fellows 20F20319, JSPS.
Y.\! Liu is supported by Grant-in-Aid for Early Career Scientists 20K14355 and 22K13954, JSPS.
%
%

\end{document}